\documentclass[12pt,a4paper,reqno]{amsart}
\usepackage{xypic,a4wide,amsmath,amssymb,amsthm,mathtools}
\mathtoolsset{showonlyrefs}
\usepackage[mathcal]{euscript}
\usepackage{mathrsfs}

\let\mathcal=\mathscr

\newtheorem{thm}{Theorem}[section]

\newtheorem{lemma}[thm]{Lemma}
\newtheorem{prop}[thm]{Proposition}
\newtheorem{defin}[thm]{Definition}

\theoremstyle{remark}
\newtheorem{rem}[thm]{Remark}
\newtheorem{ex}[thm]{Example}
\newenvironment{remark}{\begin{rem}\rm}{\qee\end{rem}}
\newenvironment{example}{\begin{ex}\rm}{\qee\end{ex}}

\newcommand{\End}{\mbox{\it End}\,}
\newcommand{\cA}{{\mathcal A}}

\newcommand{\cC}{{\mathcal C}}

\newcommand{\cE}{{\mathcal E}}
\newcommand{\cF}{{\mathcal F}}
\newcommand{\cO}{{\mathcal O}}

\newcommand{\calL}{{\mathcal L}}
\newcommand{\cM}{{\mathcal M}}
\newcommand{\cN}{{\mathcal N}}
\newcommand{\cI}{{\mathcal I}}
\newcommand{\cK}{{\mathcal K}}
\newcommand{\cQ}{{\mathcal Q}}

\newcommand{\Ext}{\mbox{\rm Ext}}

\newcommand{\ati}{{{\mathcal D}_{\cE}}}

\newcommand{\Hom}{\operatorname{Hom}}
\newcommand{\gr}{\operatorname{gr}}
\newcommand{\rk}{\operatorname{rk}}
\newcommand{\Rep}{\operatorname{Rep}}
\newcommand{\kmod}{{\mbox{$\Bbbk$\mbox{\bf -mod}}}}

\newcommand{\g}{{\mathfrak g}}

\newcommand{\qee}{\mbox{\hspace{0.2mm}}\hfill$\triangle$}

\newcommand{\Der}{\operatorname{Der}}

\newcommand{\cEnd}{{\mathcal E}\!\text{\it nd}}

\newcommand{\cHom}{\mathscr{H}\!\text{\em om}}

\begin{document}
%

\begin{center}
	{\large \bf  LIE  ALGEBROID COHOMOLOGY \\[5pt] AS A DERIVED FUNCTOR} \\[30pt]
	{\sc Ugo Bruzzo}\\[10pt] \small 
	Area di Matematica,  Scuola Internazionale Superiore di Studi \\ Avanzati (SISSA), Via Bonomea 265, 
	34136 Trieste, Italy; \\[3pt] Department of Mathematics, University of Pennsylvania, \\ 209 S 33rd st., Philadelphia, PA 19104-6315, USA; \\[3pt] Istituto Nazionale di Fisica Nucleare, Sezione di Trieste
 \end{center}
 
 \bigskip
\begin{quote}
\small  {\sc Abstract.}    We show that the hypercohomology of the Chevalley-Eilenberg-de Rham complex of a Lie algebroid $\calL$ over a   scheme with coefficients in an $\calL$-module  can be expressed as a derived functor. We use this fact to study a Hochschild-Serre type spectral sequence attached to an extension of Lie algebroids.
\end{quote}

 \let\svthefootnote\thefootnote
\let\thefootnote\relax\footnote{
\hskip-\parindent {\em Date: }  Revised 14 April 2017  
Research partly supported by {\sc indam-gnsaga}. {\sc u.b.} is a member of the {\sc vbac} group.
}
\addtocounter{footnote}{-1}\let\thefootnote\svthefootnote
\setcounter{tocdepth}{1}
{\small\tableofcontents}

\bigskip\section{Introduction}

As it is well known, the  cohomology groups of a Lie algebra $\g$ over a ring $A$ with coefficients in a $\g$-module $M$ can be computed directly from the Chevalley-Eilenberg complex, or as the derived functors of the invariant submodule functor, i.e., the functor which with ever $\g$-module $M$ associates the submodule of $M$ 
$$ M^\g = \{ m \in M \,\vert\, \rho(\g)(m)=0\},$$
where $\rho\colon\g\to\End_A(M)$ is a representation of $\g$ (see e.g.\ \cite{weibel}).
In this paper we show an analogous result for the hypercohomology of the Chevalley-Eilenberg-de Rham complex of a Lie algebroid $\calL$ over a scheme $X$, with coefficients in a representation $\cM$ of $\calL$ (the notion of hypercohomology is recalled in Section \ref{gen}). An important tool for the proof of this result will be the interplay of the theory of Lie algebroids with that of Lie-Rinehart algebras, as we shall indeed use several results from Rinehart's pioneering paper \cite{Rinehart63}.

This approach to Lie algebroid hypercohomology permits us to look at the spectral sequence that one can attach, as shown in \cite{BMRT}, to an extension of Lie algebroids (a ``Hochschild-Serre'' spectral sequence), as the Grothendieck spectral sequence associated with the right derived functors of a composition of left-exact functors, affording us a better understanding of its properties. In a similar way, one can define a kind of local-to-global spectral sequence, which relates the hypercohomology of the Chevalley-Eilenberg-de Rham complex of a Lie algebroid $\calL$   with the cohomologies of the Lie-Rinehart algebras $\calL(U)$, where $U\subset X$ is an affine open subscheme.

We would like to stress that the only assumptions we make about the base scheme $X$ are that it is a noetherian separated scheme over a field, and no regularity hypothesis is made. Moreover, the same results can be obtained in the case of Lie algebroids over analytic spaces, and, in particular, complex manifolds (technically, whenever in the following an affine subscheme or affine open cover will be used, one should instead use a Stein subspace or Stein cÁover).
Thus this theory may be in principle applied to study vector fields on singular varieties, K\"ahler forms on singular Poisson varieties, coisotropic singular subvarieties of Poisson varieties, singular varieties acted on by an algebraic group, etc. 

One should mention that derived functors in connection with Lie-Rinehart algebras were also considered by J.~Huebschmann \cite{Huebsch09}.

The contents of this paper are as follows. Section \ref{gen} gives some generalities about Lie algebroids and discusses their relation with Lie-Rinehart algebras. In Section \ref{derived} we discuss our main result, the isomorphism of the Chevalley-Eilenberg-de Rham hypercohomology with the derived functors of the invariant submodule functor. Section \ref{HSss} is devoted to the study of a Hochschild-Serre spectral sequence
and of the local-to-global spectral sequence, and some related results.

\smallskip
\noindent {\bf Acknowledgments.} This paper was mostly written during a visit at the Department of Mathematics of the University of Pennsylvania in April-May 2016. I thank Penn for support, and the 
colleagues and staff at Penn for their usual warmth and hospitality, and in particular Murray Gernsternhaber and Jim Stasheff for their kind invitation to present this result in the Penn Deformation Theory Seminar. A thank-you also to my friend Volodya Rubtsov who introduced me to Lie algebroids and  Lie-Rinehart algebras, and to Ettore Aldrovandi, Geoffrey Powell and Pietro Tortella for useful conversations. Finally, I thank   the referee whose suggestions allowed me to improve the presentation.
 
\bigskip\section{Generalities on Lie algebroids }\label{gen}
     
     All schemes will be noetherian and separated.
Let $X$ be a 
scheme over a field $\Bbbk$. We shall denote by $\cO_X$ the sheaf of regular functions on $X$, by $\Bbbk_X$ the constant sheaf on $X$ with stalk $\Bbbk$, and by $\Theta_X$ the tangent sheaf of $X$ (the sheaf of derivations of the structure sheaf $\cO_X$), which is a sheaf of $\Bbbk_X$-Lie algebras.

\subsection{Lie algebroids} A   Lie algebroid $\calL$ on $X$ is a coherent $\cO_X$-module $\calL$ equipped with:
\begin{itemize}
\item a $\Bbbk$-linear Lie bracket defined on   sections of $ \calL$, satisfying the Jacobi identity;
\item a   morphism of $\cO_X$-modules $a\colon \calL \to\Theta_X$, 
called the {\em anchor} of $\calL$, which is also a morphism of sheaves of $\Bbbk_X$-Lie algebras.
\end{itemize}
The  Leibniz rule
\begin{equation}\label{leibniz} [s,ft] = f[s,t] +a(s)(f)\,t  \end{equation}
is required to hold 
for all sections $s,t$ of $\calL$ and $f$ of $\cO_X$ (actually the Leibniz rule and the Jacobi identity imply that the anchor is a morphism of $\Bbbk_X$-Lie algebras).

A morphism $(\calL,a)\to (\calL',a')$ of Lie algebroids defined over the same scheme $X$ is a morphism of $\cO_X$-modules $f\colon\calL\to\calL'$, which is compatible with the brackets defined in $\calL$ and in $\calL'$, and such that $a'\circ f=a$.  Note that this implies
that the kernel of a morphism of Lie algebroids has a trivial anchor, i.e., it is a sheaf of $\cO_X$-Lie algebras.

We  introduce the {\em Chevalley-Eilenberg-de Rham complex}  of $\calL$, which is a  sheaf of
differential graded algebras. This is  $\Lambda^\bullet_{\cO_X}\calL^\ast$ as a sheaf of $\cO_X$-modules, with a product  given by the wedge product, and  a $\Bbbk$-linear differential  $d_\calL\colon \Lambda^\bullet_{\cO_X}\calL^\ast\to \Lambda^{\bullet+1}_{\cO_X}\calL^\ast$   defined by the  formula
  \begin{eqnarray*}\label{diff}
(d_\calL\xi)(s_1,\dots,s_{p+1}) &=& 
\sum_{i=1}^{p+1}(-1)^{i-1}a(s_i)(\xi(s_1,\dots,\hat s_i,
\dots,s_{p+1})) \\ & + & \sum_{i<j}(-1)^{i+j}
\xi([s_i,s_j],\dots,\hat s_i,\dots,\hat s_j,\dots,s_{p+1})
\end{eqnarray*}
  for   $s_1,\dots,s_{p+1}$ sections of $\calL$, and $\xi$ a section of $\Lambda^p_{\cO_X}\calL^\ast$.
  The name of this complex comes of course from the fact that when $X$ is a point, $\calL$ is a Lie algebra, and the complex reduces to the Chevalley-Eilenberg complex of that Lie algebra; while when
  $\calL = \Theta_X$, and $X$ is smooth, one gets the de Rham complex of $X$.
  
  The hypercohomology\footnote{We remind the reader, if needed, that given an additive, left-exact functor between abelian categories $F \colon \mathfrak A \to \mathfrak B$, where $\mathfrak A$ is assumed to have enough injectives, the hyperfunctors
  $\mathbb H^iF\colon K_+(\mathfrak A)\to \mathfrak B$, with $i\ge 0$, are defined as 
  $\mathbb H^iF(A^\bullet) = H^i(F(I^\bullet))$, where $I^\bullet$ is a complex of injective objects quasi-isomorphic to $A^\bullet$
  (here $K_+(\mathfrak A)$   is the category of complexes of objects in $\mathfrak A$ bounded from below). When
  $\mathfrak A = \mathfrak{Qcoh}(X)$ for a scheme $X$, $\mathfrak B$ is the category of abelian groups, and $F$ is the global section functor $\Gamma$, the hyperfunctors $ \mathbb H^i\Gamma$ are called the {\em hypercohomology groups,} denoted $\mathbb H^i(X,\cF^\bullet)$ for a complex $\cF^\bullet$ of quasi-coherent $\cO_X$-modules.}
of the complex $(\Lambda^\bullet_{\cO_X}\calL^\ast,d_\calL)$ is called the {\em   Lie algebroid cohomology} of $\calL$.
On a complex manifold, if $\calL$ is locally free as an $\cO_X$-module, this cohomology is  isomorphic to the Lie algebroid cohomology of the smooth complex Lie algebroid $L$ obtained by matching (in the sense of \cite{Lu97,Mokri,GSX}) the holomorphic Lie algebroid $\calL$ with the anti-holomorphic tangent bundle $T_X^{0,1}$ \cite{GSX,BR-cohom}. Since $\rk L = \rk  \calL + \dim X$,  the hypercohomology of the complex $(\Lambda^\bullet_{\cO_X}\calL^\ast,d_\calL)$ in this case vanishes in degree higher than $\rk  \calL + \dim X$. 
        
\begin{ex}[The Atiyah algebroid]
A fundamental example of Lie algebroid is $\ati$, the {\em Atiyah algebroid} associated with a coherent $\cO_X$-module $\cE$.
This is defined as the sheaf of first order differential operators   on $\cE$   having scalar symbol, i.e., satisfying the following property: for all open subsets $U\subset X$, and every section $D\in\ati(U)$,
there is a vector field $v\in\Theta_X(U)$ such that
$$D(fs) = fD(s) + v(f) s $$
for all $f\in\cO_X(U)$ and $s\in\calL(U)$. Note that the vector field $v$ is unique.   $\ati$ sits inside
the   exact sequence of sheaves of $\cO_X$-modules
\begin{equation} \label{atiyah} 0 \to  \cEnd_{\mathcal O_X}(\cE) \to \ati \xrightarrow{\sigma } \Theta_X ,
\end{equation} where the morphisms $\sigma$ mapping $D$ to $v$, called {\it the symbol map}, plays the role of   anchor. The bracket is given by the commutator of differential operators.
Note that, when $\cE$ is locally free, the anchor is surjective so that $\ati$ is an extension of $\Theta_X$ by $\cEnd_{\cO_X} (\cE)$.
\end{ex}

\subsection{Representations} The notion of   representation of a Lie algebroid can be neatly introduced by using the Atiyah algebroid.
\begin{defin}
A {\em representation} of Lie algebroid $\calL$ on a coherent $\cO_X$-module $\cM$ is a morphism of  Lie algebroids $\rho: \calL \to \mathcal{D}_{\!\!\cM}$. \end{defin}
Thus, given a section $s$ of $\calL$, $\rho(s)$ acts on sections of $\cM$ as a first-order differential operator, and $\rho([s,t])(m) = [\rho(s),\rho(t)] (m)$, where the bracket in the right-hand side is the commutator of operators. Moreover,
$a(s)$ is required to be the symbol of $\rho(s)$ (here $a$ is the anchor of $\calL$).

We shall denote by $\Rep(\calL)$
the category of representations of the Lie algebroid $\calL$. Given two
representations $(\cM,\rho)$ and $(\cN,\beta)$, a morphism $(\cM,\rho)\to (\cN,\beta)$ is a morphism of $\cO_X$-modules $f\colon\cM\to\cN$ such that
$\beta(s) (f(m)) = f(\rho(s)(m))$ for all sections $s$ of $\calL$ and $m$ of $\cM$.
The category $\Rep(\calL)$ is abelian.

Given a  representation $(\cM,\rho)$ of $\calL$, one can introduce the twisted modules $ \cM\otimes_{\cO_X} \Lambda^k\calL^\ast$, with a differential

 \begin{eqnarray}\label{diffoid}
(d_\rho \xi)(s_1,\dots,s_{p+1}) &=& 
\sum_{i=1}^{p+1}(-1)^{i-1}\rho (s_i)(\xi(s_1,\dots,\hat s_i,
\dots,s_{p+1})) \\ & + & \sum_{i<j}(-1)^{i+j}
\xi([s_i,s_j],\dots,\hat s_i,\dots,\hat s_j,\dots,s_{p+1}).
\end{eqnarray}
The hypercohomology of the complex $(\cM\otimes_{\cO_X}\Lambda^\bullet\calL^\ast,d_\rho)$ is called the cohomology of $\calL$ with values in $(\cM,\rho)$, denoted $\mathbb H^\bullet(\calL;\cM,\rho)$ or $\mathbb H^\bullet(\calL;\cM)$.

\subsection{Lie-Rinehart algebras}
We explore now the relation of Lie algebroids over algebraic schemes with Lie-Rinehart algebras. 
\begin{defin}
Let $A$ be a unital commutative algebra over a field $\Bbbk$. A {\em $(\Bbbk,A)$-Lie-Rinehart algebra} is a pair $(L,a)$, where $L$ is an $A$-module equipped with a $\Bbbk$-linear Lie algebra bracket, and $a\colon L \to \operatorname{Der}_\Bbbk(A)$ a representation of $L$ in $\operatorname{Der}_\Bbbk(A)$  (the anchor) that satisfies the Leibniz rule \eqref{leibniz}, where now $s,t\in L$ and $f\in A$.  \end{defin}
For brevity, we shall often say that $L$ is a Lie-Rinehart algebra over $\Bbbk$, understanding the algebra $A$ and the representation $a$.  

\begin{example}
Let $M$ be an $A$-module. Given a derivation $\bar D\in \operatorname{Der}_\Bbbk(A)$,
a derivation $D$ of $M$ over $\bar D$ is a $\Bbbk$-linear morphism $D\colon M\to M$ such that
$$D(x \, m) = x \, D(m) + \bar D (x) \,m$$ 
for all $x\in A$, $m\in M$. The set  $D(M)$ of derivations of $M$ over derivations of $A$
is an $A$-module. The morphism $\sigma\colon D(M)\to\operatorname{Der}_\Bbbk(A)$,
mapping $D$ to $\bar D$, is a morphism of Lie $\Bbbk$-algebras, and $(D(M),\sigma)$ is
a  $(\Bbbk,A)$-Lie-Rinehart algebra. Actually $\sigma$ is a symbol
 map, so that $D(M)$
could be called the space of derivations of $M$ with scalar symbol.
Moreover, the sequence of $A$-modules
$$ 0 \to \operatorname{End}_A(M) \to D(M) \stackrel{\sigma}{\longrightarrow} \operatorname{Der}_\Bbbk(A) $$
is exact, and
the symbol morphism $\sigma$ is surjective when $M$ is free over $A$. 
 The pair $(D(M),\sigma)$ is  the {\em Atiyah-Lie-Rinehart algebra} of the module $M$. 
\end{example}

Again, a representation of a $(\Bbbk,A)$-Lie-Rinehart algebra $L$ is an $A$-module $M$ with a morphism of Lie-Rinehart algebras
$L\to  D(M)$. A cohomology  $H^\bullet(L;M)$ is defined by means of a differential analogous to the one in eq.~\eqref{diffoid}.

Assume that $L$ is finitely generated over $A$. 
If we take $X = \operatorname{Spec} A$ and  localize the $A$-module $L$, the resulting $\cO_X$-module $\calL$ has a Lie algebroid structure. Conversely,
given a Lie algebroid $\calL$ over a scheme $X$, for every open $U\subset X$, 
$\calL(U)$ is a  $(\Bbbk, \cO_X(U))$-Lie-Rinehart algebra, finitely generated as an $\cO_X(U)$-module. This establishes the following result. Let $\mbox{\bf{Lie-Alg}}_\Bbbk$ be the category of Lie algebroids over  schemes over $\Bbbk$, and let $\mbox{\bf{Lie-Alg}}_\Bbbk^{\mbox{\tiny aff}}$ be the full subcategory of Lie algebroids over   affine schemes. Moreover, let $\mbox{\bf{Lie-Rin}}_\Bbbk$ be the category of Lie-Rinehart algebras $L$ over the field $\Bbbk$, with $L$ finitely generated over  a $\Bbbk$-algebra $A$. Morphisms  are defined in the obvious way.

\begin{prop} The categories  $\mbox{\bf{Lie-Alg}}_\Bbbk^{\mbox{\tiny\rm aff}}$ and $\mbox{\bf{Lie-Rin}}_\Bbbk$ are equivalent. In particular, if we fix a unital commutative $\Bbbk$-algebra $A$, and take $X = \operatorname{Spec} A$, the categories  $\mbox{\bf{Lie-Alg}}(X)$   of Lie algebroids over $X$, and  $\mbox{\bf{Lie-Rin}}_{(\Bbbk,A)}$ of $(\Bbbk,A)$-Lie-Rinehart algebras, are equivalent.\footnote{$\mbox{\bf{Lie-Alg}}_\Bbbk^{\mbox{\tiny\rm aff}}$ is a fibered category over the category of affine $\Bbbk$-schemes, and analogously, $\mbox{\bf{Lie-Rin}}_\Bbbk^{\mbox{\tiny op}}$ is a fibered category over the category of unital $\Bbbk$-algebras.}
\label{equiv}
\end{prop}

 \bigskip      
\section{Cohomology as derived functor}\label{derived}

\subsection{The universal enveloping algebroid} \label{univalgebroid}
In \cite{Rinehart63}, the universal enveloping algebra $\mathfrak U(L)$ of a $(\Bbbk,A)$-Lie-Rinehart algebra $L$ was introduced.  Let us recall the definition. One gives the direct sum $\bar L  = A\oplus L$ a structure of Lie algebra over $\Bbbk$ by letting 
$$ [f \oplus x, g \oplus y ] = \left(a(x)(g) - a(y)(h)\right)\oplus [x,y],$$
where $a$ is the anchor of $L$.
Let us denote by $\mathfrak U(\bar L )$ the usual enveloping algebra of $\bar L $, and denote 
by $\iota\colon \bar L  \to \mathfrak U(\bar L )$ the natural morphism. Let $\mathfrak U(\bar L )^+$ the subalgebra generated by the image of $\iota$. The universal    enveloping algebra $\mathfrak U(L)$ of $L$ is the quotient 
$\mathfrak U(\bar L )^+/J(L)$, where $J(L)$ is the ideal generated by the elements
$$ \iota(f u ) -\iota(f)\iota(u)$$
for $f\in A$ and $u \in \bar L $. Note that in $\mathfrak U(L)$ the identity
$$ s \cdot f - f \cdot s = a(s)(f) $$
holds for every $f\in A$ and $s\in L$.

The induced map  $\iota \colon A \to \mathfrak U(L)$ turns out to be a monomorphism,
and gives $\mathfrak U(L)$   both a left and right $A$-module structure.  
Moreover, the anchor $a$ gives $A$ a structure of $\mathfrak U(L)$-module, induced by
$$ g\cdot f = gf,\qquad s \cdot f = a(s)(f)$$
for $s\in L$ and $g,f\in A$. This allows one to define a morphism (augmentation map) $\varepsilon\colon\mathfrak U(L)\to A $
by mapping $u \in \mathfrak U(L)$ to $ u \cdot 1$.

The universal enveloping algebra $\mathfrak U(L)$ satisfies a universality property \cite{Mord}, which implies that, 
given a morphism $f\colon L\to L'$ of $(\Bbbk,A)$-Lie-Rinehart algebras, there is a unique morphism $\mathfrak U(f) \colon \mathfrak U(L)\to \mathfrak U(L')$ which, when restricted to the natural image of $L$ into $\mathfrak U(L)$,  coincides with $f$. According to the discussion before Proposition \ref{equiv}, we can define the universal enveloping algebroid $\mathfrak U(\calL)$ of a Lie algebroid $\calL$ over a scheme $X$
as the sheaf of algebras whose space of sections on any affine open subset $U\subset X$ is the universal enveloping algebra $\mathfrak U(\calL(U))$ of the $(\Bbbk,\cO_X(U))$-Lie-Rinehart algebra $\calL(U)$. The functoriality of the universal enveloping algebra ensures that $\mathfrak U(\calL)$ is well defined. Each local  universal enveloping algebra has a morphism  $\mathfrak U(\calL(U))\to\cO_X(U)$ and these give rise to a morphism $\varepsilon\colon\mathfrak U(\calL) \to\cO_X $.

If $L$ is a Lie-Rinehart algebra, there is an equivalence of categories between the category of representations of $L$, and the category of  (left) $\mathfrak U(L)$-modules. Indeed if $(M,\rho)$ is a representation of $L$, then
$M$ becomes a $\mathfrak U(L)$-module by letting
$$ f \cdot m = fm, \qquad s\cdot m =  \rho(s)(m)$$
for $m\in M$, $f\in A$ and $s\in L$. 
Viceversa, if $M$ is a $\mathfrak U(L)$-module, one defines a representation $\rho$ of $L$ on $M$ by letting
$$\rho(s) (m) = s\cdot m.$$

 This immediately produces the analogous result for Lie algebroids. We shall henceforth assume that $\calL$ is a locally free Lie algebroid on a   scheme $X$ over a field $\Bbbk$. We shall denote by
 $\mathfrak U(\calL)${\bf{-mod}} the category of left $\mathfrak U(\calL)$-modules.

\begin{prop} The category $\Rep(\calL)$ of representations of a Lie algebroid $\calL$ (over $X$) is equivalent to the category $\mathfrak U(\calL)$\bf{-mod}.
\end{prop}

The category of sheaf of modules over any sheaf of rings has enough injectives, and as a consequence, $\Rep(\calL)$ has enough injectives. 
Alternatively, to prove this claim one can use the fact that  for every $x\in X$ the stalk $\calL_x$
is a $(\Bbbk,\cO_{X,x})$-Lie-Rinehart algebra, whose category has enough injectives, and mimic the proof that the category of $\cO_X$-modules has enough injectives \cite{Hartshorne}.

\subsection{Lie algebroid cohomology}
Given  a representation $(\cM,\rho)$ of a Lie algebroid $\calL$, we denote by $\cM^\calL$ its invariant submodule, i.e., for every open subset $U\subset X$,
$$\cM^\calL(U) = \{ m \in \cM(U) \,\vert\, \rho(\calL) (m) = 0  \}.$$
Note that whenever the anchor of $\calL$ is not trivial, this is not an $\cO_X$-module, but only a $\Bbbk_X$-module, as
$$\rho(s)(fm) = f\rho(s)(m) + (a(s)(f)) \,m.$$

We introduce the left-exact  additive functor
\begin{eqnarray} I^\calL \colon \Rep(\calL) &\to& \kmod \\
\cM &\mapsto & \Gamma(X,\cM^\calL).\end{eqnarray}
Since $ \Rep(\calL)$ has enough injectives, we may consider the right derived functors $R^iI^\calL$.

\begin{thm} Let $\calL$ be a locally free a Lie algebroid over the scheme $X$, and let $\Rep(\calL)$ the category of its representations.
The functors $R^iI^\calL$ are isomorphic to the hypercohomology functors
\begin{equation} \mathbb H^i(\calL;-) \colon  \Rep(\calL) \to  \kmod.\end{equation}
\label{main}\end{thm}

\begin{proof} We start by noting that the functors $\mathbb H^i(\calL;-)$, in view of the standard properties of the hypercohomology functors (i.e., the association with every exact sequence of representations of $\calL$ of a connecting morphism satisfying the usual functorial properties), make up an exact $\delta$-functor.\footnote{We recall that a $\delta$-functor between two abelian categories $\mathfrak A$ and $\mathfrak B$ is a collection of functors $\{T^i\colon\mathfrak A\to\mathfrak B, i\ge 0\}$, and for every exact sequence $0 \to A \to B \to C \to C\to 0$ in $\mathfrak A$,   morphisms $\delta_i\colon T^i(C) \to T^{i+1}(A)$ satisfying some properties: naturality with respect to morphisms of exact sequences, and exactness of the sequences
$$ ... \to T^{i-1}(C) \xrightarrow{\delta_{i-1}} T^{i}(A) \to  T^{i}(B) \to  T^{i}(C)  \xrightarrow{\delta_{i}} T^{i+1}(A) \to \dots $$
For a complete definition, and the property of $\delta$-functors we are going to use, see \cite{Tohoku,Tennison,Hartshorne}.}
To get the desired isomorphism, we need to show that $\mathbb H^0(\calL;-)\simeq I^\calL $, and moreover, that the  $\delta$-functor is {\em effaceable,} which is tantamount to saying that $\mathbb H^i(\calL;\cI)=0$
for $i>0$ whenever $\cI$ is an injective object in  $ \Rep(\calL) $. The first property will be proved in the next Lemma. For the second, we consider the homology complex of $\Bbbk_X$-modules $\{\cC^i\}_{i\ge 0}$
$$\cC^i = \mathfrak U(\calL) \otimes_{\cO_X}\Lambda^i\calL$$
with the differential
\begin{multline}\label{homcomplex}  \partial ( u \otimes x_1\wedge \dots \wedge x_i )   = 
\sum_{j=1}^i (-1)^{j-1} ux_j \otimes x_1\wedge \dots \wedge \hat x_j \wedge \dots \wedge x_i  \\
+\sum_{1\le j<k \le n} (-1)^{j+k} u\otimes [x_j,x_k] \wedge  x_1\wedge \dots \wedge \hat x_j \wedge \dots \wedge \hat x_k \wedge \dots \wedge x_i .
\end{multline}
Moreover, one  has the  (surjective) augmentation morphism $\varepsilon\colon \cC^0 =  \mathfrak U(\calL) \to \cO_X$, cf.~Section \ref{univalgebroid}.
We note that, for every $x\in X$, the stalk $\calL_x$ is a   $(\Bbbk,\cO_{X,x})$-Lie-Rinehart algebra, projective  over $\cO_{X,x}$, and that the 
complex $\{\cC^\bullet\}$  induces a complex $$\cC^\bullet_x =  \mathfrak U(\calL_x) \otimes_{\cO_{X,x}}\Lambda^i\calL_x$$
which is a projective resolution of the local ring $\cO_{X,x}$ \cite{Rinehart63}. Therefore, $\{\cC^\bullet\}$ is a  resolution
of $ \cO_X$ (as a $ \mathfrak U(\calL)$-module; note that the differential $\partial$ in \eqref{homcomplex} is $ \mathfrak U(\calL)$-linear with respect to the argument $u$).  If we apply the functor $\Hom_{\mathfrak U(\calL)}(-,\cI)$ to $\cC^\bullet$, 
the cohomology complex we obtain is exact as $\cI$ is injective, and on the other hand, since
$$\mathcal H\!om_{\mathfrak U(\calL)}( \mathfrak U(\calL) \otimes_{\cO_X}\Lambda^i\calL,\cI) \simeq \cI\otimes_{\cO_X}\Lambda^\bullet\calL^\ast,$$
 it coincides with the complex
$\{\Gamma(X,\cM\otimes_{\cO_X}\Lambda^\bullet\calL^\ast),d_\rho\}$, so that $\mathbb H^i(\calL;\cI)=0$
for $i>0$ (note that the sheaves $\cI\otimes_{\cO_X}\Lambda^\bullet\calL^\ast$ are injective $\mathfrak U(\calL)$-modules).
\end{proof}

\begin{lemma} Given a representation $(\cM,\rho)$ of $\calL$, there is an isomorphism $$\mathbb H^0(\calL;\cM) \simeq\Gamma(X,\cM^\calL) = I^\calL(\cM) .$$
\end{lemma} 
\begin{proof} One has
$$ \mathbb H^0(\calL;\cM) = \ker \delta_0 \colon \Gamma(X,\cI^0) \to  \Gamma(X,\cI^1),$$
where $\{\cI^\bullet,\delta_\bullet\}$ is a complex of injective $\Bbbk_X$-modules quasi-isomorphic to $(\cM\otimes_{\cO_X}\Lambda^\bullet\calL^\ast,d_\rho)$. Since $\Gamma$ is left-exact, and the two complexes are quasi-isomorphic, we have
\begin{multline} \ker \delta_0 \colon \Gamma(X,\cI^0) \to  \Gamma(X,\cI^1)   \simeq  \Gamma(X,\ker \delta_0 \colon \cI^0\to \cI^1) \\ \simeq 
\Gamma(X,\ker d_\rho \colon \cM \to   \cM\otimes_{\cO_X} \calL^\ast) = \Gamma(X,\cM^\calL) .\end{multline}
\end{proof} 

\subsection{Lie algebroid cohomology as Ext groups}
As for Lie algebras,  Lie algebroid cohomology can be written in terms of Ext groups. 
If $(L,a)$ is a $(\Bbbk,A)$-Lie-Rinehart algebra,  
the algebra $A$ can be made into a representation of $L$ by defining $\rho\colon L \to D_A \simeq A \oplus \Der_\Bbbk(A)$ as
$$\rho(x)(\alpha) = a(x)(\alpha).$$
Then one easily checks that, if $M$ is a reprentation of $L$, 
$$\Hom_L(A,M) \simeq M^L,$$
where the isomorphism maps $f \in \Hom_L(A,M)$ to $f(1)$. On the other hand, 
the equivalence between the categories of representations of $L$ and of $\mathfrak U(L)$-modules
implies 
$$\Hom_{\mathfrak U(L)}(A,M) \simeq \Hom_L(A,M) \simeq  M^L$$
where $A$ is a $\mathfrak U(L)$-module via the augmentation morphism $\mathfrak U(L)\to A$.
If $\calL$ is a Lie algebroid on a scheme $X$, and $\cM$ is a representation of $\calL$, 
this yields an isomorphism of sheaves of $\Bbbk_X$-modules
$$\cHom_{\mathfrak U(\calL)}(\cO_X,\cM) \simeq \cM^{\calL}$$
and taking global sections
$$\Hom_{\mathfrak U(\calL)}(\cO_X,\cM) \simeq   I^{\calL} (\cM),$$ 
so that one has:
\begin{prop} Let $\calL$ be a locally free a Lie algebroid over $X$. The functors    $ \Ext^i_{\mathfrak U(\calL)}(\cO_X,-)$  and $\mathbb H^i(\calL;-)$ are isomorphic as functors $\Rep(\calL)\to \Bbbk$\bf{-mod}.
 \end{prop}

\bigskip
\section{A Hochshild-Serre spectral sequence}\label{HSss}

Let
\begin{equation}\label{extens}
0 \to \cK \to \calL \to \cQ \to  0
\end{equation}
be an exact sequence of locally free Lie algebroids over $X$. 
We shall associate with such an extension a spectral sequence analogous to the Hochschild-Serre spectral sequence \cite{Hoch-Serre53}; this is actually a Grothendieck spectral sequence associated with the right derived functors of the composition of two left-exact functors. This   spectral sequence was  already introduced in \cite{BMRT} in a different way. We shall compare the two definitions later on in this section. The new viewpoint will allow us to study this spectral sequence in an easier way, gaining a better grasp of it.

\subsection{The Hochshild-Serre spectral sequence as a Grothedieck spectral sequence}
As we already noticed, $\cK$ is a sheaf of $\cO_X$-Lie algebras, i.e.,
it has a vanishing anchor. Thus, if $\cM$ is a representation of $\calL$, 
the $\cK$-invariant submodule  $\cM^\cK$ is an  $\cO_X$-module, and moreover,
it is a representation of $\cQ$. One has a commutative diagram of functors
\begin{equation}\label{functors}
\xymatrix{
\Rep(\calL) \ar[r]^{(-)^\cK} \ar[rd]_{I^\calL} & \Rep(\cQ) \ar[d]^{I^\cQ} \\
& \kmod}\  .
\end{equation}
Then we know \cite{Tohoku} that, if $(-)^\cK$  maps injective objects of $\Rep(\calL)$ to
$I^\cQ$-acyclic objects of $\Rep(\cQ)$,
 for every representation $\cM$ of $\calL$ there is a spectral sequence converging to $R^\bullet I^\calL(\cM)$, whose second 
page is $E_2^{pq} = R^pI^\cQ(R^q\cM^\cK)$.

So we need to prove:

\begin{lemma}
If $\cI$ is an injective object of $\Rep(\calL)$, then $R^pI^\cQ(\cI^\cK)=0$ for $p>0$.
\end{lemma}

\begin{proof} Actually one proves a stronger result, namely, that $(-)^\cK$  maps injective objects 
of $\operatorname{Rep}({\mathscr L})$ to injective objects $\operatorname{Rep}({\mathscr Q})$. This follows from the
fact that the functor  $(-)^\cK$ is right adjoint to the natural forgetful functor  $\operatorname{Rep}({\mathscr Q})\to \operatorname{Rep}({\mathscr L})$, which is exact.
 \end{proof} 
  
We shall denote the sheaves $R^q\cM^\cK$ as $\mathcal H^q(\cK;\cM)$.
 For every $q$, $\mathcal H^q(\cK;\cM)$ is a coherent $\cO_X$-module, and is a representation of $\cQ$.
 Note that $\mathcal H^0(\cK;\cM)\simeq \cM^{\cK}$.

Se, rephrasing our claim about the spectral sequence, we have:

\begin{thm} Let $\calL$ be a locally free a Lie algebroid over $X$. For every representation $\cM$ of $\calL$ there is a spectral sequence $E$ converging to  $\mathbb H^{\bullet}(\calL;\cM)$, whose second page is 
\begin{equation}\label{2ndpageAlg} E_2^{pq}=\mathbb H^p(\cQ;\mathcal H^q(\cK;\cM)).\end{equation}
\label{ss1}
\end{thm}

It may be useful to record the explicit form of the five-term sequence of this spectral sequence:

\begin{multline}
0 \to \mathbb H^1(\cQ;\cM^{\cK}) \to \mathbb H^1(\calL;\cM) \to  \\  \mathbb H^0(\cQ;\mathcal H^1(\cK;\cM)) \to \mathbb H^2(\cQ; \cM^{\cK}) \to \mathbb H^2(\calL;\cM) \,.
\end{multline} 

\subsection{Another approach to the spectral sequence} We define a filtration $\{\cF^q_p\}$ of the complex $\cM^\bullet=\cM\otimes_{\cO_X}\Lambda^\bullet\calL^\ast$ by defining $\cF^q_p$ as the subsheaf of $\cM\otimes_{\cO_X}\Lambda^q\calL^\ast$ whose sections are annihilated by the wedge product of $q-p+1$ sections of $\cK$. The graded object of this filtration is  
\begin{equation}\label{graded}
\operatorname{gr}_p \cM^{q} =
\cF^{q}_{p}/\cF^{q}_{p+1} \simeq\cM\otimes_{\cO_X} \Lambda^p\cQ^\ast\otimes_{\cO_X}\Lambda^{q-p}\cK^\ast.\end{equation}
According to the mechanism shown in \cite[Ch.~0, 13.6.4]{EGA3-I}, the filtration $\{\cF^q_p\}$ induces
a filtration of the complex that computes the hypercohomology of the complex $\cM\otimes_{\cO_X}\Lambda^\bullet\calL^\ast$, thus giving rise to a spectral sequence $\widetilde E$ which converges to $\mathbb H^\bullet(\calL;\cM)$. This spectral sequence is the Lie algebroid analogue of the Hochschild-Serre spectral sequence as it was originally defined \cite{Hoch-Serre53}, and it is therefore interesting to compare it with the spectral sequence we have previously introduced. 

Let us briefly sketch this construction.   By standard homological techniques (see e.g.\ \cite{Tennison} --- basically, the horseshoe Lemma),  one can introduce injective resolutions $\mathcal C^{q,\bullet }$ of $\cM^q=\cM\otimes_{\cO_X}\Lambda^q\calL^\ast$, and $F_p\mathcal C^{q,\bullet }$ of $\cF_p^q$, such that $F_p\mathcal C^{q,\bullet }$ is a filtration of $\mathcal C^{q,\bullet }$, and $\gr_p \mathcal C^{q,\bullet}= F_p\mathcal C^{q,\bullet }/ F_{p+1}\mathcal C^{q,\bullet }$ is an injective resolution of $\gr_p \cM^q$.
We consider the total complex 
\[
T^k = \bigoplus_{p+q=k} \Gamma(X,\mathcal C^{p,q})
\]
whose cohomology is the hypercohomology of $\Omega_\cA^\bullet$. Its descending filtration is defined by
\begin{equation}\label{ss2}     
F_\ell T^k = \bigoplus_{p+q=k} \Gamma(X, F_\ell \mathcal C^{p,q}).
\end{equation}
       
\begin{lemma} \label{quotient} 
For all $\ell$, $p$, $q$ there is  an isomorphism
 \[
 \Gamma(X,F_\ell \mathcal C^{p,q}) /
 \Gamma(X,F_{\ell +1}\mathcal C^{p,q})\simeq \Gamma(X,\operatorname{gr}_\ell \mathcal C^{p,q}).
 \]
 \end{lemma}
 \begin{proof} As $F_p \mathcal C^{q,\bullet }/ F_{p+1}\mathcal C^{q,\bullet } \simeq 
 \gr_\ell \mathcal C^{p,q}$ one has the exact sequences
\[ 
 0 \to   \Gamma(X, F_{\ell +1}\mathcal C^{p,q} ) \to \Gamma(X, F_\ell \mathcal C^{p,q}) \to \Gamma(X, \gr_\ell\mathcal C^{p,q}) \to 
H^1(X, F_{\ell +1} \mathcal C^{p,q} )=0.
\]
\end{proof}

 As a consequence, the zeroth  term of the spectral sequence given by the  filtration $F_\ell T^k$  is
 \[
\widetilde E_0^{\ell k}  = \bigoplus_{p+q=k+\ell} \Gamma(X, \gr_\ell\mathcal C^{p,q}) .
 \]
 Recalling that the differential $d_0 \colon \widetilde E_0^{\ell k} \to \widetilde E_0^{\ell,k+1}$ is induced by the differential of the complex $\cM^\bullet$, we obtain 
 \[
\widetilde E_1^{\ell k} \simeq \mathbb H^k(X,\cF_\ell^\bullet / \cF_{\ell + 1}^\bullet).
 \]
By plugging in the isomorphism  \eqref{graded},  we have
\begin{equation}\label{1stpageAlg}\widetilde E_1^{pq} \simeq \mathbb H^q(\cK;\cM\otimes_{\cO_X}\Lambda^p\cQ^\ast).\end{equation}
Note that, since $\cQ$ is a  $\cK$-module, the sheaves $\Lambda^p\cQ^\ast$ are $\cK$-modules as well.

\begin{thm} Let $E$ be the spectral sequence of Theorem \ref{ss1}, and let $\widetilde E$ be the spectral sequence associated with the filtration \eqref{ss2}. The   spectral sequences $E$ and $\widetilde E$ are  isomorphic.  \label{ssai}
\end{thm} 
To prove this theorem we use a construction by Grothendieck \cite[\S 2.5]{Tohoku}. 
Let $F\colon\mathfrak C\to \mathfrak C'$ be a left-exact additive functor between abelian categories, and assume that  $\mathfrak C$ has enough injectives. A {\em resolvent functor} for $F$ is an exact functor $\mathbb F\colon\mathfrak C\to K_{\ge 0}(\mathfrak C')$ (where $K_{\ge 0}(\mathfrak C')$ is the category of complexes of objects in $\mathfrak C'$ with vanishing terms in negative degree) with an augmentation morphism $F\to\mathbb F$  such that
\begin{enumerate} \item $F \to \ker \,(\mathbb F^0 \to \mathbb F^1)$ is an isomorphism; \item If $I$ is an injective object in $\mathfrak C$, then $H^i(\mathbb F (I)^\bullet)=0$ for $i>0$.
\end{enumerate} 

One has \cite[Prop.~2.5.3]{Tohoku}:

\begin{prop} Let $A^\bullet\in K_{\ge 0}(\mathfrak C)$. The first/second spectral sequence of the double complex $\mathbb F(A^\bullet)$ is functorially isomorphic to the first/second spectral sequence for the hypercohomology of  $F$ with respect to the complex $A^\bullet$. \label{groth1}
\end{prop}

\noindent{\em Proof of Theorem \ref{ssai}.} We note that the functor 
\begin{align}   \mathbb F\colon \Rep(\cQ) & \to    K_{\ge 0}(\kmod) \\
\cM & \mapsto  \Gamma(X, \cM\otimes  \Lambda^\bullet \cQ^\ast )
\end{align}
is a resolvent functor for $I^{\cQ}$. Let    $\cM$  be an object in $\Rep(\calL)$,
and let $\cI^\bullet$ be an injective resolution of $\cM$. 
We apply  Proposition \ref{groth1} with 
$$F=I^{\cQ}\colon\Rep(\cQ)\to\kmod \qquad\mbox{and}\qquad A^\bullet = (\cI^\bullet)^{\cK}.$$
The double complex $C=\mathbb F(A^\bullet)$ is
$$C^{p,q} = \Gamma(X,(\cI^q)^{\cK}\otimes_{\cO_X}\Lambda^p \mathfrak \cQ^\ast).$$
The first and second pages of the second spectral sequence are given by \eqref{1stpageAlg} and
\eqref{2ndpageAlg}, respectively. The $d_1$ differential of the second spectral sequence of $C$  coincides with that of the spectral sequence $\widetilde E$ (they essentially are the differential of the Chevalley-Eilenberg complex of $\cQ$), so that the two   spectral
sequences are isomorphic. On the other hand, the $\cQ$-modules $(\cI^\bullet)^{\cK}$ are $I^\cQ$-acyclic, so that the hypercohomology
of $I^\cQ$ with respect to the complex $ (\cI^\bullet)^{\cK}$ is
$$ H^i(F((\cI^\bullet)^{\cK})) = H^i(I^\calL(\cI^\bullet))) = 
 \mathbb  H^i(\calL;\cM).$$
So the spectral sequences $\widetilde E$ and $E$ coincide, and Theorem \ref{ssai} is proved.\qed

\begin{remark} Proposition \ref{groth1} also provides a straightforward, albeit less transparent, proof of 
 our Theorem \ref{main}. One sets
$$ \mathfrak C = \Rep(\calL), \quad
\mathfrak C' = \kmod, \quad F = I^\calL, \quad \mathbb F^i(\cM) = \Gamma(X,\cM\otimes_{\cO_X} \Lambda^i\calL^\ast)$$
and takes for $A^\bullet$ an injective resolution $\cI^\bullet$ of a representation $\cM$ of $\calL$.
The second spectral sequence of the double complex $\mathbb F(\cI^\bullet)$ degenerates
and converges to the hypercohomology $\mathbb H^\bullet (\calL;\cM)$, while
the hypercohomology of $F$ with respect to the complex $\cI^\bullet$ is given by the right derived functors
$RI^\calL(\cM)$.
\end{remark}

\begin{remark} Evidently these constructions define a spectral sequence attached to an exact sequence of Lie-Rinehart algebras
(the case of Lie-Rinehart algebras obtained as global sections of  a $C^\infty$   Lie algebroid was treated in \cite{MK87}, while 
the case of ideals in the Atiyah-Lie-Rinehart algebra of a module was treated in 
 \cite{Roub80}).  So, if 
 $$ 0 \to K \to L \to Q \to 0$$
 is an exact sequence of $(\Bbbk,A)$-Lie-Rinehart algebras, $K$ is an ideal in $L$, and $M$ is an $L$-module, there is a spectral sequence with terms in $\Bbbk${\bf -mod} whose first two pages are
 $$E_1^{p,q} = \Lambda^p Q^\ast \otimes H^{q}(K;M),\qquad 
E_2^{p,q} =  H^p(Q;H^{q}(K;M)).$$
\end{remark}

\subsection{A local-to-global spectral sequence} If $\calL$ is a Lie algebroid on a scheme $X$, the functor $$I^\calL\colon\Rep(\calL)\to\kmod, \qquad \cM\to\Gamma(X,\cM^\calL), $$ is itself a composition: 
$$  \Rep(\calL) \xrightarrow{(-)^\calL }\Bbbk_X\mbox{\bf -mod} \xrightarrow{\Gamma} \kmod.$$
The derived functors of $(-)^\calL $ are
$$R^i(-)^\calL \colon\Rep(\calL) \to \Bbbk_X\mbox{\bf -mod},\qquad \cM \mapsto \mathcal H^i(\calL;\cM)$$
and when $\cI$ is an injective object in $\Rep(\calL)$, one has $\mathcal H^i(\calL;\cI)=0$ for $i>0$ by the same argument we have already used. As a result, if $\cM$ is  a representation of $\calL$, there is a spectral sequence, converging to $\mathbb H^\bullet(\calL;\cM)$, whose second term is
$$E_2^{pq} = H^p(X,\mathcal H^q(\calL;\cM)).$$
This is a kind of local-to-global spectral sequence for Lie algebroid cohomology. Indeed, 
the sheaf $\mathcal H^i(\calL;\cM)$ is the sheafification of the presheaf which with an open set $U\subset X$ associates the degree $i$ cohomology of the $(\Bbbk,\cO_X(U)$-Lie-Rinehart algebra $\calL(U)$ with coefficients in $\cM(U)$.

As a particular case, we get a generalized de Rham theorem
\cite{BR-cohom}.

\begin{thm} Let $\calL$ be a locally free a Lie algebroid over $X$, and let $\cM$ be a representation of $\calL$.
Assume the complex $\cM^\bullet=\cM\otimes_{\cO_X}\Lambda^\bullet\calL^\ast $ is exact in positive degree,
i.e., it is a resolution of a $\Bbbk_X$-module $\mathcal C$. Then
$$\mathbb H^i(\calL;\cM ) \simeq H^i(X,\mathcal C),\qquad i\ \ge 0.$$
\end{thm}
\begin{proof} One has  $\mathcal H^i(\calL;\cM)$ $=0$ for $i>0$, and then 
the above spectral sequence degenerates at $E_2$.
\end{proof}

\subsection{An extension to complexes} Actually Proposition \ref{groth1} also allows one to prove an extension of Theorem \ref{main} to complexes.
Given a locally free Lie algebroid $\calL$ on the scheme $X$, we denote $\operatorname{Krep}(\calL) = K_{\ge 0}(\Rep(\calL))$, while
$K(\Bbbk)$ will be the category of complexes of $\Bbbk$-vector spaces, again with vanishing terms in negative degree. Then the functor $I^\calL$ extends to a functor
$$ I^\calL \colon \operatorname{Krep}(\calL) \to K(\Bbbk)$$
and we have hyperfunctors $\mathbb R^i I^\calL \colon \operatorname{Krep}(\calL) \to \Bbbk\hbox{{\bf -mod}}$.

 \begin{thm} Let $\calL$ be a locally free a Lie algebroid over $X$.  For every complex $\cM^\bullet$ in $ \operatorname{Krep}(\calL)$ there are natural isomorphisms of functors $ \operatorname{Krep}(\calL) \to \Bbbk\hbox{{\bf -mod}}$
 $$ \mathbb R^i I^\calL (\cM^\bullet) \simeq \mathbb H^i(X,\mathcal T^\bullet),\qquad i \ge 0$$
 where $\mathcal T^\bullet$ is the total complex of the double complex  $\cM^\bullet \otimes_{\cO_X} \Lambda^\bullet\calL^\ast$.
\end{thm}
\begin{proof} The two spectral sequences in Proposition \ref{groth1} converge to $ \mathbb RI^\calL (\cM^\bullet)$ and $ \mathbb H(X,\mathcal T^\bullet)$, respectively.
\end{proof}
 
 \bigskip 
\frenchspacing

\def\cprime{$'$} \def\cprime{$'$} \def\cprime{$'$} \def\cprime{$'$}

\end{document}